\documentclass[11pt,english]{article}
\usepackage{amsmath, amsthm, latexsym, amssymb,url,dsfont}
\usepackage{amsfonts, bbold}
\usepackage{epsfig}
\usepackage{enumerate}
\usepackage{graphicx}
\usepackage{yhmath}
\usepackage{mathdots}
\usepackage{mathtools} 
\usepackage{pifont}

\usepackage{tikz}
\usetikzlibrary{matrix,arrows,decorations.pathmorphing}
\usepackage{tikz-cd}

\usepackage{relsize}

\newcommand{\shrinkmargins}[1]{
  \addtolength{\textheight}{#1\topmargin}
  \addtolength{\textheight}{#1\topmargin}
  \addtolength{\textwidth}{#1\oddsidemargin}
  \addtolength{\textwidth}{#1\evensidemargin}
  \addtolength{\topmargin}{-#1\topmargin}
  \addtolength{\oddsidemargin}{-#1\oddsidemargin}
 \addtolength{\evensidemargin}{-#1\evensidemargin}
  }

\shrinkmargins{.4}

\theoremstyle{plain}

\newtheorem{theorem}{Theorem}[section]
\newtheorem{corollary}[theorem]{Corollary}

\newtheorem{lemma}[theorem]{Lemma}
\newtheorem{proposition}[theorem]{Proposition}

\newtheorem*{teo}{Theorem}

\newtheorem{definition}[theorem]{Definition}

\theoremstyle{remark}
\newtheorem{remark}[theorem]{Remark}

\theoremstyle{definition}
\newtheorem{example}[theorem]{Example}

\def \Z { \mathbb{Z}}
\def \Q { \mathbb{Q}}

\def \Gal { \text{Gal}}
\def \ker { \text{Ker}}
\def \R { \mathbb{R}}

\def \tr { {\rm Tr}}
\def \Tr { {\rm Tr}}

\newcommand{\oo}{{\mathfrak{o}}}

\newcommand{\QQ}{{\mathbb{Q}}}
\newcommand{\ZZ}{{\mathbb{Z}}}

\DeclareMathOperator{\lcm}{lcm}
\begin{document}

\thispagestyle{empty}
\setcounter{tocdepth}{7}

\title{The Shape of cyclic number fields}
\author{Wilmar Bola\~nos \and Guillermo Mantilla-Soler}

\date{}

\maketitle

\begin{abstract}

Let $m>1$  and $\mathfrak{d} \neq 0$ be  integers such that $v_{p}(\mathfrak{d}) \neq m$ for any prime $p$. We construct a matrix $A(\mathfrak{d})$  of size $(m-1) \times (m-1)$ depending on only of $\mathfrak{d}$  with the following property: For any tame $\Z/m\Z$-number field $K$ of discriminant $\mathfrak{d}$ the matrix $A(\mathfrak{d})$ represents the Gram matrix of the integral trace zero form of $K$. In particular, we have that the integral trace zero form of tame cyclic number fields is determined by the degree and discriminant of the field. Furthermore, if in addition to the above hypotheses, we consider real number fields, then the shape is also determined by the degree and the discriminant.\\ 
\end{abstract}

\section{Introduction}

Let $K$ be a number field of degree $n:=[K:\mathbb{Q}]$  and let $\oo_{K}$ be its maximal order. The {\it trace zero module} of $\oo_{K}$ is the $\Z$-submodule of $\oo_{K}$ given by the Kernel of the trace map i.e., $\oo_{K}^{0}=K^{0} \cap \oo_{K}$ where $K^{0}:= \{ x \in K : \tr_{K/\mathbb{Q}}(x)=0 \}.$ The {\it integral trace-zero form} of $K$ is the isometry class of the rank $n-1$ quadratic $\Z$-module $\langle \oo_{K}^{0}, {\rm Tr}_{K/\Q} \rangle$ given by restricting the trace pairing from $\oo_{K} \times \oo_{K}$  to $\oo_{K}^{0} \times \oo_{K}^{0}.$  For $K$ of degree $n=1,2$ it is clear, by checking the discriminant, that the isometry class of the quadratic module $\langle \oo_K^{0} , \Tr_{K / \QQ} ()|_{\oo_K}\rangle$ determines the field $K$. For general degrees this is not the case, see for instance \cite[\S 3]{Man}. However, recently the second named author and one of his coauthors have shown that for square free discriminants, with some extra conditions on the signature and degree, the integral trace-zero form determines the conjugacy class of the field:

\begin{theorem}\cite[Theorem 2.13]{M4}. Suppose that $K$ is totally real, of fundamental discriminant $d$ and such that $\gcd(n,d)=1$. If $(\Z/n\Z)^{*}$ is cyclic then $\left<\oo_K^{0} , \Tr_{K / \QQ} ()|_{\oo_K}\right>$  is a complete invariant for $K$. In other words, for any number field  $L$ we have that \[ \left<\oo_K^{0} , \Tr_{K / \QQ} ()|_{\oo_K}\right> \simeq  \left<\oo_L^{0} , \Tr_{L / \QQ} ()|_{\oo_L} \right>  \  \mbox{if and only if} \ K \simeq L.  \]

\end{theorem}

In contrast to the above result in the case of Galois number fields of  prime degree $\ell =n$ the integral trace-zero form does not discriminante beyond the discriminant. Moreover, in such a case, the integral trace-zero form is isometric to the root lattice $\mathbb{A}_{\ell-1}$ times a constant defined solely in terms of the discriminant.  Recall that for a given positive integer $m$ the root lattice $\mathbb{A}_{m}$ is the $m$-dimensional lattice associated to the quadratic form \[\sum_{1 \leq i  \leq m} 2x_{i}^{2} - \sum_{\substack{1\leq i,j \leq m \\ |i-j| = 1}} x_{i} x_{j} \] or equivalently with Gram matrix in some basis given by \[ \mathcal{A}_{m}:=\begin{pmatrix} 2 & - 1 & 0 & 0 & \dots & 0 \\ - 1 & 2 & -1 & 0 & \cdots & 0 \\ 0 & -1 & 2 & -1 & \dots & 0 \\ \vdots & \vdots & \ddots & \ddots &  \ddots & \vdots \\ 0 & 0 & 0 & -1 & 2 & -1 \\ 0 & 0 & 0 & 0 & -1 & 2  \end{pmatrix}\]

 Given an integer $d$ let ${\rm rad}(d)$ be the square free integer that has the same signature and same prime factors as $d$.

\begin{theorem}\cite[Theorem 2.9]{M3}. Let $\ell$ be a prime and let $K$ be a $\Z/\ell\Z$-number field of discriminant $\mathfrak{d}(K)$. Suppose that $\gcd(\ell, \mathfrak{d}(K))=1$. Then, the Gram matrix of $\left<\oo_K^{0} , \Tr_{K / \QQ} ()|_{\oo_K}\right>$ with respect to some basis  is equal to  \[{\rm rad}(\mathfrak{d}(K)) \mathcal{A}_{\ell -1}.\] In particular,  if $L$ is a $\Z/\ell\Z$-number field with $\gcd(\ell, \mathfrak{d}(L))=1$ then  \[\left<\oo_K^{0} , \Tr_{K / \QQ} ()|_{\oo_K}\right> \simeq \left<\oo_L^{0} , \Tr_{L / \QQ} ()|_{\oo_L}\right> \ \mbox{if and only if} \   \mathfrak{d}(K)=  \mathfrak{d}(L). \]

\end{theorem}
\begin{remark}
Note that the condition $\gcd(\ell, \mathfrak{d}(K))=1$ is equivalent to say that the field $K$ is tame, i.e., that there is no rational prime that is wildly ramified in $K$.
\end{remark}

The purpose of this paper is to generalize the result above to general cyclic number fields of arbitrary degree. Our main theorem is the following:

\begin{teo}[cf. Theorem \ref{main}]

Let $m \neq 1$ be a positive integer and let $K$ be a $\Z/m\Z$-number field of discriminant $\mathfrak{d}(K)$. Suppose that $K$ is tame. There exists a matrix $A(\mathfrak{d}(K)) \in {\rm M}_{(m-1) \times (m-1)}(\Z)$ depending only on  $\mathfrak{d}$ such that the Gram matrix of \[\left<\oo_K^{0} , \Tr_{K / \QQ} ()|_{\oo_K}\right>\] with respect to some basis  is equal to \[A(\mathfrak{d}(K)).\] In particular,  if $L$ is a tame $\Z/m\Z$-number field then  \[\left<\oo_K^{0} , \Tr_{K / \QQ} ()|_{\oo_K}\right> \simeq \left<\oo_L^{0} , \Tr_{L / \QQ} ()|_{\oo_L}\right> \ \mbox{if and only if} \   \mathfrak{d}(K)=  \mathfrak{d}(L). \]

\end{teo}

\begin{remark}

In the case that $m=\ell$ is a prime number, $A(\mathfrak{d}(K))={\rm rad}(\mathfrak{d}(K))\mathcal{A}_{\ell -1}$. Hence the construction here directly generalizes the results in \cite{M3}.

\end{remark}

\subsection{The Shape}

Another quadratic invariant, with a more geometric interpretation and closely related to the trace zero form, that has been studied by several authors is the shape of $K$.  Endow $K$ with the real-valued  $\mathbb{Q}$-bilinear form $b_K$ whose associated quadratic form is given by \[b_K(x,x):= \sum_{\sigma : K \hookrightarrow \mathbb{C}}  |\sigma(x)|^2.\] The \textit{shape} of $K$, denoted $\textrm{Sh}(K)$, is  the isometry equivalence class of $ (\oo_ K^{\bot},b_K)$ up to scalar multiplication, where $\oo_ K^{\bot}$ is the image of $\oo_K$ under the projection map, 
 $\alpha \mapsto \alpha_{\bot}:=n\alpha-\tr_{K/\mathbb{Q}}(\alpha)$, i.e.,
\[\oo_ K^{\bot}:= \{\alpha_\bot : \alpha \in \oo_K \}=(\mathbb{Z}+n \oo_K) \cap \oo_{K}^{0}.\]
Thus $\textrm{Sh}(K)=\textrm{Sh}(L)$ if and only if $(\oo_K^{\bot},b_K) \simeq  (\oo_L^{\bot}, \lambda b_L) $ for some $\lambda \in \mathbb{R}^{*}$. Equivalently,  ${\rm Sh}(K)$ can be thought as the  $(n-1)$-dimensional lattice inside $\R^{n}$, via the Minkowski embedding, that is the orthogonal complement of $1$ and that is defined up to reflection, rotations and scaling by $\R^{*}$. Hence {\rm Sh}(K) corresponds to an element to the {\it space of shapes} \[\mathcal{S}_{n-1}:= {\rm GL}_{n-1}(\Z) \setminus {\rm GL}_{n-1}(\R)/{\rm GO}_{n-1}(\R).\] The distribution of shapes of number fields in $\mathcal{S}_{n}$ have been the subject  of interesting current research (see \cite{bhargavaPh,RobH,RobH1,M3}).  It turns out that for cyclic real number fields  the Shape is determined by the discriminant, moreover:

\begin{teo}[cf. Theorem \ref{TheShape}]
Let $m$ be a positive integer and let $K$ and $L$ two totally real tame $\Z/m\Z$-number fields. Then, the following are equivalent: 

\begin{itemize}

\item[(a)] $\left<\oo_K , \Tr_{K / \QQ} ()\right> \simeq  \left<\oo_L , \Tr_{L / \QQ} () \right>.$

\item[(b)] $(\oo_K^{\bot},b_K) \simeq  (\oo_L^{\bot}, b_L).$ 

\item[(c)] $\left<\oo_K^{0} , \Tr_{K / \QQ} ()|_{\oo_K}\right> \simeq \left<\oo_L^{0} , \Tr_{L / \QQ} ()|_{\oo_L}\right>.$

\item[(d)] $\mathfrak{d}(K)=  \mathfrak{d}(L)$.

\end{itemize}

\end{teo}

\section{The trace zero module}
In this section, we study the behavior of the trace zero module of cyclic number fields tamely ramified and the connection with their discriminant. The main goal is to extend some results about trace zero modules developed in \cite{Man} and \cite{M3}.\\

\begin{definition}
Let $K$ be a number field and $\oo_K$ its maximal order. The trace zero module $\oo_K^{0}$ is defined by $$ \oo_K^{0}:= \{x \in \oo_K : \Tr_{K / \QQ}(x) = 0 \}. $$
\end{definition}

\begin{lemma}\label{0basis}
Let $K$ be a tamely ramified cyclic number field of degree $m$. Let $\sigma \in \Gal(K / \QQ)$  be a generator of the Galois group. Suppose  $B:=\{ \mathbb{e}_{1},..., \mathbb{e}_{m}\}$  is a normal integral basis for $\oo_K$, i.e., an integral basis such that $\sigma(\mathbb{e}_{m}) =\mathbb{e}_{1}$  and $\sigma(\mathbb{e}_{j}) =\mathbb{e}_{j+1}$ for  $1 \leq j \leq m-1$. Then $B_{0}:=\{ \mathbb{e}_1 - \mathbb{e}_2, \mathbb{e}_2 - \mathbb{e}_3, \dots, \mathbb{e}_{m-1} - \mathbb{e}_{m}\}$ is an integral basis for $\oo_K^{0}$.
\end{lemma}

\begin{proof}
Since $B$ is a normal basis each element $\mathbb{e}_j - \mathbb{e}_{j+1}$  of $B_{0}$ belongs to $\oo_K^{0}$. On the other hand, thanks to  Hilbert's 90, if $u \in K$ is such that $\Tr_{K / \QQ}(u) = 0$ then there exist $b \in K$ such that $u = b - \sigma(b)$. Therefore, if $u \in \oo_K^{0}$ then
\begin{eqnarray*}
 u  & = & b - \sigma(b) \\
 & = & (c_1 \mathbb{e}_1 + c_2 \mathbb{e}_2 + \dots + c_{m} \mathbb{e}_{m}) - \sigma(c_1 \mathbb{e}_1 + c_2 \mathbb{e}_2 + \dots + c_{m}\mathbb{e}_{m}) \\
 &= & (c_1 - c_{m})\mathbb{e}_1 + (c_2 - c_1)\mathbb{e}_2 + \dots + (c_m - c_{m-1})\mathbb{e}_{m},
\end{eqnarray*}
for some $c_i \in \QQ$.  Since $u \in \oo_K$ we must have $c_{j+1} - c_j \in \ZZ$ for $j = 0, 1, \dots, (m-1)$, where we define $c_0 := c_{m}$. Moreover, since

\[c_i - c_{m} = \sum_{j = 1}^{i} (c_j - c_{j-1}) \in \ZZ\] by rearranging we obtain
\begin{eqnarray*}
u  & = & (c_1 - c_{m})\mathbb{e}_1 + (c_2 - c_1)\mathbb{e}_2 + \dots + (c_m - c_{m-1})\mathbb{e}_{m} \\
  & = & (c_1 - c_{m})(\mathbb{e}_1 - \mathbb{e}_2) + (c_2 - c_{m})(\mathbb{e}_2 - \mathbb{e}_3) + \dots + (c_{m-1} - c_{m})(\mathbb{e}_{m-1} - \mathbb{e}_{m}).
\end{eqnarray*}
By ranks  we conclude that $\{ \mathbb{e}_1 - \mathbb{e}_2, \mathbb{e}_2 - \mathbb{e}_3, \dots, \mathbb{e}_{m-1} - \mathbb{e}_{m} \}$ is an integral basis for $\oo_{K}^0.$
\end{proof}

\subsection{A Gram matrix representation of the trace zero form}

In this section we use a trace zero basis, coming from a normal integral basis as in the previous section, to find a canonical Gram matrix for the integral trace zero form.\\

\begin{lemma}{\label{3.2.2}}Let $K$ be a number field of degree $m$ and let $G_K := \ZZ + \oo_K^{0}$. We have
\[ \left|\oo_K / G_K \right| = \left|\Tr_{K / \QQ} (\oo_K) / m\ZZ \right|.\]
\end{lemma}

\begin{proof}
By the isomorphism theorem on groups the result follows from:

\begin{itemize}
\item The group $G_K$ is a subgroup of  $\oo_K$ that contains the Kernel of the trace map,
\item the image of $G_K$ under the trace is equal to $m\ZZ$.
\end{itemize}
\end{proof}

\begin{corollary}{\label{3.2.3}} If $K$ is a tamely ramified number field of degree $m$, then
$$ \left|\oo_K / G_K \right| = m.$$
\end{corollary}

\begin{proof}
Since $K$ is tamely ramified, by \cite[Corollary 5 to Theorem 4.24]{Nar} \[\Tr_{K/ \QQ}(\oo_K) = \ZZ.\]
Using this, the result follows from Lemma \ref{3.2.2}.
\end{proof}

\begin{theorem}{\label{zerotrace1}}
Let $K$ and $L$ be tamely ramified cyclic number fields of degree $m$. Then,
\[ \left<\oo_K^{0} , \Tr_{K / \QQ} ()|_{\oo_K}\right> \simeq  \left<\oo_L^{0} , \Tr_{L / \QQ} ()|_{\oo_L} \right> 
\ \mbox{ if and only if} \  \  \mathfrak{d}(L) = \mathfrak{d}(K). \]
\end{theorem}

\begin{proof}
Suppose that  $ \left<\oo_K^{0} , \Tr_{K / \QQ} ()|_{\oo_K}\right> \simeq  \left<\oo_L^{0} , \Tr_{L / \QQ} ()|_{\oo_L} \right> $. Since the decomposition of $G_{K}=\ZZ + \oo_K^{0}$  (resp $G_{L}=\ZZ + \oo_L^{0}$)  is an orthogonal decomposition with respect to the trace pairing we have the following equalities between determinants of the trace in the respective modules:
\[\mathfrak{d}(G_K) = \mathfrak{d}(\oo_K^{0}) m  = \mathfrak{d}(\oo_L^{0})   m = \mathfrak{d}(G_L).\]
Since
\[\mathfrak{d}(G_K) =\left| \oo_K / G_K \right|^2 \mathfrak{d}(\oo_K) \  \mbox{and} \  \mathfrak{d}(G_L) =\left| \oo_L / G_L \right|^2 \mathfrak{d}(\oo_L) \]
the result follows thanks to Corollary \ref{3.2.3}.\\

\noindent On the other hand, if $\mathfrak{d}(K) = \mathfrak{d}(L)$ then by \cite[Theorems 4.2 and 4.5]{MB} there exist normal integral bases  $B:=\{ \mathbb{e}_{1},..., \mathbb{e}_{m}\}$  and  $B':=\{ \mathbb{e}'_{1},..., \mathbb{e}'_{m}\}$  of $\oo_K$ and $\oo_L$, respectively, such that   $\left<\oo_K , \Tr_{K / \QQ} ()\right> $ and $  \left<\oo_L , \Tr_{L / \QQ} () \right>$ are isometric via an isometry  \[\gamma: \oo_K \to \oo_L\] such that $\gamma(\mathbb{e}_{i})=\mathbb{e}'_{i}$ for all $1 \leq i \leq m$. It follows from Lemma \ref{0basis} that such an isometry $\gamma$ restricts  to an isometry between the quadratic modules $\left<\oo_K^{0} , \Tr_{K / \QQ} ()|_{\oo_K}\right>$ and $ \left<\oo_L^{0} , \Tr_{L / \QQ} ()|_{\oo_L} \right>$. 
\end{proof}

\begin{definition}
Let $m$ be a positive integer.  For every $d$ a positive divisor of $m$  we let $A_d$ be  an $m \times m $ matrix defined by
\[ (A_d)_{i,j} :=  \left\{ \begin{array}{ll} 1 & \mbox{ if } \frac{m}{d} \big| (i-j) \\
& \\
0 & {\rm otherwise}.  \end{array} \right.\]

\end{definition}

Suppose that $K$ is a tame cyclic number field of degree $m$.  Theorems 4.2 and 4.5 of \cite{MB} state that  there exists a normal integral basis $B$ of $\oo_K$  such that  the Gram matrix of
$\left<\oo_K , \Tr_{K / \QQ} ()\right> $ in such basis is equal to  \[\sum_{d|m} a_d A_d,\]
where $a_d$ are integers  given in Lemma 4.3 of \cite{MB} that only depend on the discriminant of $K$. For the reader's convenience see also Definition \ref{LosCoeficientes}.\\  

We devote the rest of this section to describe a similar canonical decomposition, i.e., a Gramm matrix depending only on the discriminant of the field, for the trace zero module $\left<\oo^{0}_K , \Tr_{K / \QQ} ()\right> $.\\

Let $n \ge 2$ be an integer . Then, we define $\mathbb{B}_n$ as the $n \times n$ matrix

\[\begin{pmatrix}  1 & 0 & 0& \dots & -1 \\ -1 & 1 & 0 & \dots & 0 \\ 0 & -1 & 1 & \dots & 0 \\  \vdots & 0 & \ddots & \ddots & \vdots \\ 0& 0 & \cdots &  -1 & 1 \\ \end{pmatrix}.\] In other words,

\[ (\mathbb{B}_n)_{i,j}:= \left\{ \begin{array}{rl}
     1 & \mbox{ if }  i=j  \\
     -1 & \mbox{ if } i=j+1 \ {\rm or} \ (i,j) =(1,n)\\
     0 & \mbox{ otherwise. }
\end{array} \right. \]

\begin{proposition}\label{MatrizRestCero}
Let $m$ be an integer bigger than $1$ and let $K$ be a degree $m$ number field. Let $M$ be the Gram Matrix of $ \left<\oo_K , \Tr_{K / \QQ} ()\right> $ with respect to $\{ \mathbb{e}_{1},..., \mathbb{e}_{m}\}$, an integral basis of$\oo_K$. Let $M_0$ be the Gram matrix of $\left<\oo_K^{0} , \Tr_{K / \QQ} ()|_{\oo_K^{0}} \right> $ with respect to $\{\mathbb{e}_1 - \mathbb{e}_2, \mathbb{e}_2 - \mathbb{e}_3, \dots, \mathbb{e}_{m-1}- \mathbb{e}_{m} \}$. Then, $M_0$ is equal to $(\mathbb{B}_m^T M \mathbb{B}_m)_{(1,1)}$, the $(m,m)$ minor of $\mathbb{B}_m^T M \mathbb{B}_m$ (the matrix obtained by deleting the last row and column).  
\end{proposition}

\begin{proof}

Let $i,j$ be integers such that $1\leq i,j < m$. Then, 
\begin{eqnarray*}
(\mathbb{B}_m^T M \mathbb{B}_m)_{i,j} &=& \sum_{k=1}^{m}\sum_{l=1}^{m}\mathbb{B}_{k,i}M_{k,l}\mathbb{B}_{l,j}  \\  
&=&M_{i,j}-M_{i,j+1}-M_{i+1,j}+ M_{i+1, j+1} \\
&= &\Tr_{K / \QQ} (\mathbb{e}_{i}\mathbb{e}_{j}) -\Tr_{K / \QQ} (\mathbb{e}_{i}\mathbb{e}_{j+1})-\Tr_{K / \QQ} (\mathbb{e}_{i+1}\mathbb{e}_{j})+\Tr_{K / \QQ} (\mathbb{e}_{i+1}\mathbb{e}_{j+1}) \\
&= & \Tr_{K / \QQ}\left( (\mathbb{e}_{i}-\mathbb{e}_{i+1})(\mathbb{e}_{j}-\mathbb{e}_{j+1})\right) \\
&= & (M_{0})_{i,j}.
\end{eqnarray*}
\end{proof}

\begin{corollary}\label{ElUltimoCoro}
Let $K$ be a tame cyclic number field of degree $m \neq 1$. There exists a basis of $\oo_K^{0}$ such that the Gramm matrix of  $\left<\oo_K^{0} , \Tr_{K / \QQ} ()|_{\oo_K^{0}} \right> $ with respect to such basis is equal to the $(m,m)$ minor of \[ \sum _{d | m} a_d \mathbb{B}_m^T A_d \mathbb{B}_m,\]where the coefficients $a_d$ are given in Lemmas 3.7, 3.10 and 4.3 of \cite{MB} and depend solely on the discriminant of $K$. See also Definition \ref{LosCoeficientes} below.
\end{corollary}

\begin{proof}
Since 
\[ M = \sum _{d | m} a_d A_d \] the corollary is an immediate consequence of Proposition \ref{MatrizRestCero}, and \cite[ Theorems 4.2 and 4.5]{MB}.
\end{proof}

\subsubsection{The matrix $A(\mathfrak{d})$.}

Given a prime $p$ we denote by $v_{p}$  the usual $p$-adic valuation on the rationals.   

\begin{definition}
Suppose that $m>0$ and $\mathfrak{d}$ are integers and let $p$ be a prime. Suppose that $m \neq v_{p}(\mathfrak{d})$. The $p$-ramification index of $\mathfrak{d}$  is the rational number defined by \[e_{p}(\mathfrak{d}):=\frac{m}{m-v_{p}(\mathfrak{d})}.\]
\end{definition}

\begin{remark}
The above definition is motivated by the following fact: If $K$ is a Galois number field of degree $m$ and discriminant $\mathfrak{d}$, then $e_{p}(\mathfrak{d})$ is the ramification index of $p$ in $K$ for any prime $p$ that is not wildly ramified in $K$.
\end{remark}

\begin{definition} Let $m>1$ and $\mathfrak{d} \neq 0$ be integers. Let ${\rm div}(\mathfrak{d})$ be the set of prime divisors of  $\mathfrak{d}$. Let  $1 = d_1 < d_2 < \dots < d_{\tau(m)} = m$ be the set of positive divisors of $m$. Let
\[ P(m) := \left\{ \left(d_2^{\varepsilon_2}, d_3^{\varepsilon_3}, \dots , d_{\tau(m)}^{\varepsilon_{\tau(m)}} \right) \in \ZZ^{\tau(m)-1} : \varepsilon_i  \in \{0,1\}  \ \mbox{for all} \ i  \right\}\]
and for every $\vec{v} \in P(m)$ let
\[ \lcm(\vec{v}) := \lcm\left[ d_2^{\varepsilon_2}, d_3^{\varepsilon_3}, \dots , d_{\tau(m)}^{\varepsilon_{\tau(m)}}\right],\]
\[ \gcd(\vec{v}) := \gcd\left(d_2^{\varepsilon_2}, d_3^{\varepsilon_3}, \dots , d_{\tau(m)}^{\varepsilon_{\tau(m)}}\right),\]
and for every  divisor $d$ of $m$ let
\[P_d := \{ \vec{v} \in P(m) : \lcm(\vec{v}) = d \}. \]

If $d >1$ is a divisor of $m$, let  \[\mathbb{P}_d:= \{ p \in {\rm div}(\mathfrak{d})  : e_{p}(\mathfrak{d}) = d \}\] and \[w_{d}:=\prod_{p \in \mathbb{P}_d}p \ \mbox{and} \ f_d := \frac{w_d - 1}{d}. \]

\end{definition}

\begin{definition}\label{LosCoeficientes}

Let $m>0$ and $\mathfrak{d}$ be integers. For $d$, a divisor of $m$ not equal to $1$, we let \[a_d := \sum\limits_{\vec{v} \in P_d} \left(\gcd(\vec{v}) \prod\limits_{\varepsilon_i = 0} w_{d_i} \prod\limits_{\varepsilon_j = 1}(- f_{d_j})  \right)  \] and  let $a_1 :=  \prod\limits_{p \in {\rm div}(\mathfrak{d})} p $. Let $A(\mathfrak{d})$ be the $(m-1)\times(m-1)$ matrix defined by 
\[A(\mathfrak{d}):=\sum_{ \substack{d \mid m \\ d \neq m}}a_{d}\widehat{A}_{d},\] where $\displaystyle \widehat{A}_{d}:= \left(\mathbb{B}_m^T A_d \mathbb{B}_m \right)_{(m,m)}.$

\end{definition}

\begin{example}\label{Nueve}

Let $m=9$ and $\mathfrak{d}= 9644443241083841416681= 7^{6}\cdot13^{6}\cdot19^{8}.$ \\
Here the set of positive divisors of $m$ is $1<3<9$ and ${\rm div}(\mathfrak{d})= \{7,13, 19\}$. Hence, $P(9)$ is the subset of $\ZZ^{2}$  given by \[P(9)=\{ (1,1), (1,9), (3,1) ,(3, 9)\},\] thus \[P_{3}= \{ (3,1)\} \ \mbox{and} \ P_{9}= \{ (1,9), (3,9)\}.\]

 On the other hand $e_{7}(\mathfrak{d})=3, e_{13}(\mathfrak{d})=3$ and $e_{19}(\mathfrak{d})=9$. Thus, $\mathbb{P}_3=\{7,13\}$ and $\mathbb{P}_9=\{19\}.$ Therefore, $w_{3}=91, f_{3}=30$ and $w_{9}=19, f_{9}=2$. From this we calculate that 
 \[a_{3}= 1\cdot 19 \cdot (-30) =-570\ \mbox{and} \  a_{9}=1\cdot 91 \cdot (-2) + 3\cdot 1 \cdot (60)=-2, \] moreover $a_{1}=7\cdot13\cdot19=1729.$ Since $A_{1}$ is the identity matrix of dimension $9$, we have that \[A(\mathfrak{d})= 1729 (\mathbb{B}_{9}^{t}\mathbb{B}_{9})_{(9,9)}-570(\mathbb{B}_{9}^{t}A_{3}\mathbb{B}_{9})_{(9,9)}.\] In other words, 
 
 \[A(\mathfrak{d})=  \begin{pmatrix} \ 2318 & -1159 &  \ 570 & -1140  & \ 570 &  \ 570  & -1140  & \ 570 \\  -1159 & \ 2318 &-1159  & \ 570 &-1140  & \ 570  & \ 570 &-1140 \\  \ 570 & -1159 & \ 2318 &-1159 &  \ 570 &-1140 &  \ 570  & 570\\  
 -1140  & \ 570 &-1159 & \ 2318 &-1159  & \ 570 &-1140  & \ 570 \\ \ 570 &-1140 &  \ 570 &-1159 & \ 2318 &-1159 &  \ 570 &-1140 \\   \ 570 &   \ 570 &-1140 &  \ 570 &-1159 & \ 2318 &-1159 & \ 570 \\ -1140 &  \ 570 &  \ 570 &-1140 &  \ 570 &-1159 & \ 2318 &-1159 \\   \ 570 &-1140 &  \ 570 &  \ 570 &-1140 &  \ 570 &-1159 & \ 2318 \end{pmatrix} \]
\end{example}

\begin{example}
Suppose that $m=\ell$ a prime number. In such case \[A(\mathfrak{d})=a_{1}\widehat{A}_{1}={\rm rad}(\mathfrak{d})\left(\mathbb{B}_{\ell}^T A_1 \mathbb{B}_{\ell} \right)_{(\ell,\ell)}={\rm rad}(\mathfrak{d})\left(\mathbb{B}_{\ell}^T \mathbb{B}_{\ell} \right)_{(\ell,\ell)}={\rm rad}(\mathfrak{d})\mathcal{A}_{\ell -1}\]

\end{example}

Now we are in position to state and prove one of the main theorems in the paper:

\begin{theorem}\label{main}

Let $m$ be an integer not equal to $1$ and let $K$ be a tame $\Z/m\Z$-number field of discriminant $\mathfrak{d}(K)$. There exists a $\ZZ$-basis of $\oo_K^{0}$ such that the Gram matrix of \[\left<\oo_K^{0} , \Tr_{K / \QQ} ()|_{\oo_K}\right>\] with respect to such basis  is equal to \[A(\mathfrak{d}(K)).\] In particular,  if $L$ is a tame $\Z/m\Z$ number field then  \[\left<\oo_K^{0} , \Tr_{K / \QQ} ()|_{\oo_K}\right> \simeq \left<\oo_L^{0} , \Tr_{L / \QQ} ()|_{\oo_L}\right> \ \mbox{if and only if} \   \mathfrak{d}(K)=  \mathfrak{d}(L). \]

\end{theorem}

\begin{proof}

For each $d$ divisor of $m$ let  \[\widehat{A}_{d}:= \left(\mathbb{B}_m^T A_d \mathbb{B}_m \right)_{(m,m)} \] be the $(m-1) \times (m-1)$ matrix obtained by erasing the last row and column of $\mathbb{B}_m^T A_d \mathbb{B}_m$.  By definition, all the entries in $A_m$ are equal to $1$, hence the columns of  $\mathbb{B}_m$ belong to $\ker(A_m)$.  In particular, $\mathbb{B}_m^T A_m \mathbb{B}_m =0$ and thus the result follows from Corollary \ref{ElUltimoCoro} and  \cite[Theorems 4.2, 4.5]{MB}.

\end{proof}

\begin{example}

There are four $\ZZ/9\ZZ$-number fields with discriminant $\mathfrak{d}=7^{6}\cdot13^{6}\cdot19^{8}$(see John Jone's data base \cite{Jones}). Since 3 is unramified in any of those fields, neither of the fields have wild ramification. Such fields are defined respectively by the following polynomials: 

\begin{itemize}

\item[$\bullet$] $x^9 - x^8 - 578x^7 - 1855x^6 + 87155x^5 + 310749x^4 - 4599958x^3 - 6198626x^2 + 102071235x - 169800379$.

\item[$\bullet$] $x^9 - x^8 - 578x^7 - 1855x^6 + 87155x^5 + 518229x^4 - 2594318x^3 - 22409730x^2 - 36985319x - 7889903.$

\item[$\bullet$] $x^9 - x^8 - 578x^7 + 1603x^6 + 88884x^5 - 430992x^4 - 3668027x^3 + 27283459x^2 - 40339579x - 7447279.$

\item[$\bullet$] $x^9 - x^8 - 578x^7 + 1603x^6 + 88884x^5 - 119772x^4 - 5379737x^3 - 3169418x^2 + 113584646x + 256187183.$

\end{itemize}

Let $A(\mathfrak{d})$ be the matrix calculated in Example \ref{Nueve}. It follows from Theorem \ref{main} that for each of those fields there is a basis of the trace zero integral module such that all the Gram matrices of the trace form in that basis are  equal to the matrix  $A(\mathfrak{d})$.  This can be verified computationally, for instance in MAGMA \cite{magma}, using the code found in \cite[\S3.1]{M2}.

\end{example}

\subsection{Another explicit description.}
 
 In this subsection we show that if we apply the results obtained here to the case $m=\ell$, a prime, we recover the formulas obtained in \cite{M3}. We also see how some of the polynomial descriptions of the trace zero form of \cite{M3} are extended to the general case of this paper.\\

Let $n$ be a positive integer bigger than $1$. The extended $n$-dimensional $\mathcal{A'}_n$ lattice is the lattice associated to the matrix

\[ \mathcal{A'}_n := \mathbb{B}_n^T \mathbb{B}_n = \begin{pmatrix} \ 2 & - 1 & \ 0 & \ 0 & \dots & -1 \\ - 1 & \ 2 & -1 & \ 0 & \cdots & \ 0 \\ \ 0 & -1 & \ 2 & -1 & \dots & \ 0 \\ \vdots & \vdots & \ddots & \ddots &  \ddots & \vdots \\ \ 0 &   \ 0 & \ 0 & -1 & \ 2 & -1 \\ -1 & \ 0 & \ 0 & \ 0 & -1 & \ 2  \end{pmatrix}.\]

For any positive integer $n$ we denote by $1_{n}$ the $n\times n$ matrix with all its entries equal to $1$ and by ${\rm I}_{n}$ the identity matrix of dimension $n$

\begin{lemma}\label{Kronecker} Let $m$ be an integer bigger than $1$ and let $d\neq m$ be a positive divisor of $m$. Then,

\[ \mathbb{B}_m^T A_d \mathbb{B}_m = 1_{d}\otimes \mathcal{A'}_{\frac{m}{d}}\]

\end{lemma}

\begin{proof}
First notice that for  every $d|m$  we have $\displaystyle A_d = 1_{d} \otimes {\rm I}_{\frac{m}{d}}.$  After doing block multiplication, the result follows from the definition of $\mathcal{A'}_n$.
\end{proof}

Applying Lemma \ref{Kronecker} to the situation of Theorem \ref{main} we deduce that the Gram matrix $M_0$ of $\left<\oo_K^{0} , \Tr_{K / \QQ} ()|_{\oo_K^{0}} \right> $ in the given basis has the form

\[M_0  = \left( \sum_{\substack{ d|m \\ d<m}} a_d 1_{d} \otimes \mathcal{A'}_{\frac{m}{d}}. \right)_{(m,m)} \eqno{(2.1)}\]

If $m$ is equal to a prime $\ell$ then the above equation is simply

\[M_0 =  (a_1 \mathcal{A'}_{\ell} )_{(\ell, \ell)}.\] 
Since $a_1 = {\rm rad}(\mathfrak{d}(K))$ and $(\mathcal{A'}_{\ell})_{(\ell,\ell)}$ is the usual $(\ell-1)$-dimensional root lattice $\mathcal{A}_{\ell-1}$, the equation $\displaystyle M_0 =  (a_1 \mathcal{A'}_{\ell} )_{(\ell,\ell)}$ is precisely \cite[Theorem 2.9]{M3}.\\

An explicit polynomial description of the integral trace zero form is the following:

\begin{corollary}{\label{tracezero2}} Let $K$ be a tame totally real cyclic number field of degree $m>1$. Then,  

\[\left<\oo_K^{0} , \Tr_{K / \QQ} ()\Big|_{\oo_K^{0}} \right> \] is the lattice associated to the form \[\sum_{1 \leq i\leq j \leq m-1} c_{i,j} x_ix_j\]
where, \[  c_{i,j}:=
     \sum_{\substack{ d \mid m, d \neq m \\ m \big| (i-j)d }} 2a_d - \sum_ {\substack{ d \mid m, d \neq m \\ m \big| (|i-j| -1)d }} a_d.   \]
\end{corollary}

\begin{proof}

Note that for every $d|m$ and $d \neq m$ we have
\[ 1_d \otimes \mathcal{A'}_{\frac{m}{d}} := \left\{ \begin{array}{rl}
     2 & \mbox{ if }  (i-j) \equiv 0 \mod (\frac{m}{d})  \\
     -1 & \mbox{ if } |i -j| \equiv 1 \mod (\frac{m}{d}) \\
     0 & \mbox{ otherwise. }
\end{array} \right. \]
Then the result follows from formula $(2.1)$.

\end{proof}

\subsection{Equivalences on the shape of cyclic number fields.}

Finally we show that for real cyclic fields, that have no wild ramification, the shape of the field is characterized by the discriminant. More specifically we show:

\begin{theorem}\label{TheShape}
Let $m$ be a positive integer and let $K$ and $L$ be two totally real tame $\Z/m\Z$-number fields. Then, the following are equivalent: 

\begin{itemize}

\item[(a)] $\left<\oo_K , \Tr_{K / \QQ} ()\right> \simeq  \left<\oo_L , \Tr_{L / \QQ} () \right>.$

\item[(b)] $(\oo_K^{\bot},b_K) \simeq  (\oo_L^{\bot}, b_L).$ 

\item[(c)] $\left<\oo_K^{0} , \Tr_{K / \QQ} ()|_{\oo_K}\right> \simeq \left<\oo_L^{0} , \Tr_{L / \QQ} ()|_{\oo_L}\right>.$

\item[(d)] $\mathfrak{d}(K)=  \mathfrak{d}(L)$.

\end{itemize}

\begin{proof}
The equivalence between (c) and (d) is Theorem \ref{zerotrace1}, however this equivalence also follows from Theorem \ref{main}. The equivalence between (a) and (d) follows from \cite[Theorems 4.2 and 4.5]{MB}. Since $K$ and $L$ are totally real the bilinear forms $b_{K}$ and $b_{L}$ are just the corresponding trace forms. Hence, $(b) \Rightarrow (d)$ follows from \cite[Lemma 2.1]{M4}. To check the missing implication we see that $(a) \Rightarrow (b)$ follows from \cite[Lemma 5.1]{Casimir}
\end{proof}

\end{theorem}

{\footnotesize Wilmar Bola\~nos, Department of Mathematics, Universidad de los
Andes,  Bogot\'a, Colombia\\ (\texttt{wr.bolanos915@uniandes.edu.co})}

\noindent
{\footnotesize Guillermo Mantilla-Soler, Department of Mathematics, Universidad Nacional de Colombia,
Medell\'in, Colombia. ({\tt gmantelia@gmail.com})}

\end{document}